\pdfoutput=1
\documentclass[reqno, 11pt]{amsart}
\usepackage{amsthm, amsmath, amssymb,booktabs,xcolor,graphicx,cite}
\usepackage{tikz,pgfplots}
\usepackage[margin=1in]{geometry}
\numberwithin{equation}{section}

\usepackage[pdfstartview=FitV, 
pagebackref=false]{hyperref}

\usepackage[color,final]{showkeys} 

\colorlet{refkey}{orange!20}
\colorlet{labelkey}{blue!30}

\newtheorem{theorem}{Theorem}[section]
\newtheorem{proposition}[theorem]{Proposition}
\newtheorem{lemma}[theorem]{Lemma}

\newtheorem{corollary}[theorem]{Corollary}
\newtheorem{conjecture}[theorem]{Conjecture}
\newtheorem*{question*}{Question}

\newtheorem{fact}[theorem]{Fact}

\theoremstyle{definition}

\newtheorem*{definition*}{Definition}

\theoremstyle{remark}
\newtheorem*{remark*}{Remark}

\newcommand{\abs}[1]{\left\lvert#1\right\rvert}
\newcommand{\norm}[1]{\left\lVert#1\right\rVert}

\newcommand{\uq}{\underline{q}}
\newcommand{\oq}{\overline{q}}
\newcommand{\ur}{\underline{r}}
\newcommand{\ovr}{\overline{r}}

\newcommand{\x}{\times}

\newcommand{\EE}{\mathbb{E}}

\newcommand{\RR}{\mathbb{R}}
\newcommand{\PP}{\mathbb{P}}

\newcommand{\cG}{\mathcal{G}}
\newcommand{\cP}{\mathcal{P}}

\newcommand{\LT}{\mathsf{LT}}
\newcommand{\UT}{\mathsf{UT}}
\newcommand{\BIP}{\mathrm{BIP}}

\author{Yufei Zhao}
\address{Department of Mathematics\\ MIT\\ Cambridge, MA 02139.}
\email{yufeiz@mit.edu}
\thanks{The author is supported by a Microsoft Research PhD Fellowship.}

\title{On the lower tail variational problem for random graphs}

\begin{document}

\maketitle

\begin{abstract}
  We study the lower tail large deviation problem for subgraph counts
  in a random graph. Let $X_H$ denote the number of copies of $H$ in
  an Erd\H{o}s-R\'enyi random graph $\mathcal{G}(n,p)$. We are
  interested in estimating the lower tail probability $\mathbb{P}(X_H \le
  (1-\delta) \mathbb{E} X_H)$ for fixed $0 < \delta < 1$.

  Thanks to the results of Chatterjee, Dembo, and Varadhan, this large
  deviation problem has been reduced to a natural variational problem
  over graphons, at least for $p \ge n^{-\alpha_H}$ (and conjecturally
  for a larger range of $p$). We study this variational problem and
  provide a partial characterization of the so-called ``replica
  symmetric'' phase. Informally, our main result says that for every
  $H$, and $0 < \delta < \delta_H$ for some $\delta_H > 0$, as $p \to
  0$ slowly, the main contribution to the lower tail probability comes
  from Erd\H{o}s-R\'enyi random graphs with a uniformly tilted edge
  density. On the other hand, this is false for non-bipartite $H$ and
  $\delta$ close to 1.
\end{abstract}

\section{Background} \label{sec:background}

We consider large deviations of subgraph counts in
Erd\H{o}s-R\'enyi random graphs. Fix a graph $H$, and let
$X_H$ denote the number of copies of $H$ in an Erd\H{o}s--R\'enyi
random graph $\cG(n,p)$. For a fixed $\delta > 0$, the problem is to
estimate the probabilities
\begin{align*}
  \text{(upper tail)} \quad & \PP(X_H \ge (1 + \delta) \EE X_H) \quad \text{and}\\
  \text{(lower tail)} \quad & \PP(X_H \le (1 - \delta) \EE X_H).
\end{align*}
This problem has a long history (see \cite{Cha12} and its
references). For the order of the logarithm of the tail probability, the upper tail
problem is considered more difficult and it was resolved only fairly
recently~\cite{Cha12,DK12}. We are now interested in the finer
question of determining the large deviation rate, or equivalently the
first order asymptotics of the logarithm of the tail
probability.


Chatterjee and Varadhan~\cite{CV11} (the dense setting, with $p$ constant)
and more recently Chatterjee and Dembo~\cite{CD} (the sparse setting,
with $p \to 0$ and $p \ge n^{-\alpha_H}$ for some $\alpha_H > 0$)
showed that this large deviation problem reduces to a natural variational
problems in the space of graphons, which are a certain type of graph
limits. We begin by reviewing this connection, and then we shift our attention to analyzing the variational problem.

The language of graph limits is used throughout our
discussion, so let us review some terminologies. We refer the
readers to the beautifully written monograph by Lov\'asz~\cite{Lov12}
or the original sources, e.g., \cite{BCLSV08,BCLSV12,LS06,LS07},
for more on the subject. A \emph{graphon} is a symmetric measurable
function $W \colon [0,1]^2 \to [0,1]$ (here symmetric means $W(x,y) =
W(y,x)$). We write $V(H)$ and $E(H)$ to mean the vertex and edge set
of a graph $H$, respectively, and $v(H) = \abs{V(H)}$ and $e(H) =
\abs{E(H)}$ to denote their cardinalities. For any graphs $H$ and $G$,
we write $\hom(H, G)$ to denote the number of graph homomorphisms from
$H$ to $G$. We denote by $t(H, G) := \hom(H, G) / v(G)^{v(H)}$ the
\emph{$H$-density} in $G$. The $H$-density of a graphon $W$ is
defined by (here $W$ could be $\RR$-valued):
\[
t(H, W) := \int_{[0,1]^{V(H)}} \prod_{ij \in E(H)} W(x_i,x_j)
\prod_{i \in V(H)} dx_i.
\]
As usual, $K_t$ denotes the complete graph on $t$ vertices. As an example, we have
\[
t(K_3, W) = \int_{[0,1]^3} W(x,y)W(x,z)W(y,z) \, dxdydz.
\]
The notion of \emph{cut distance} is mentioned a few times in this
paper, but it is not used in a substantial way, so we refer the readers to
\cite[Chapter 8]{Lov12}
for details.

We write
\[
I_p(x) := x \log \frac{x}{p} + (1-x) \log \frac{1-x}{1-p}
\]
for the relative entropy function. For any function $f$, we write
$\EE[f(W)] := \int_{[0,1]^2} f(W(x,y)) \, dxdy$.

\medskip

We begin with a review of what is known for upper tails. In the dense case, for
fixed $0 < p \le q < 1$, it was shown in \cite{CV11} that as $n \to \infty$,
\begin{equation} \label{eq:ldvar}
\log \PP( t(H, \cG(n,p)) \geq q^{e(H)} ) = - (1+o(1)) \frac{n^2}{2}
\UT_p(H, q)\, ,
\end{equation}
where $\UT_p(H, q)$, for any graph $H$, is given by the \emph{upper tail
  variational problem}:
\begin{equation} \label{eq:UTp}
  \UT_p(H, q) :=
  \left\{
  \begin{array}{l}
    \text{minimize } \EE[I_p(W)] \\
    \text{subject to } t(H, W) \geq q^{e(H)}.
  \end{array}\right.
\end{equation}
Here $W$ is taken over all graphons.  We shall use $\UT_p(H, q)$ to
refer to the variational problem as well as its value, i.e.,
$\min\{\EE[I_p(W)] : t(H,W) \geq q^{e(H)}\}$ (it is known that the
minimum is always attained by some $W$; see Lemma~\ref{lem:Lagrange}
below).

Furthermore, as shown in \cite[Theorem 3.1]{CV11}, the set of minimizing $W$ in $\UT_p(H, q)$ represents the most likely models for $\cG(n,p)$
conditioned on the rare event of $t(H, \cG(n,p)) \geq
q^{e(H)}$, in the sense that the random graph conditioned on this rare
event will be exponentially more likely to be close to the minimizing
set of $W$'s in terms of cut
distance. This motiviates the study of $\UT_p(H, q)$ and related
variational problems.

We currently have few tools for solving variational problems of the
type \eqref{eq:UTp}. Note that $W \equiv q$ alway satisfies the
constraint in \eqref{eq:UTp}. We focus on the basic question:
\emph{does the constant graphon $W \equiv q$ minimize $\UT_p(H, q)$}?
The answer depends on the graph $H$ and parameters $(p, q)$. For a
fixed $H$, we wish to determine for each $(p,q)$ whether $\UT_p(H, q)
= I_p(q)$ or $\UT_p(H, q) < I_p(q)$, and in the former case, whether
the constant function $W \equiv q$ is the unique minimizer\footnote{We
  identify graphons differing on a measure zero set, as well as up to
  a measure-preserving transformation on $[0,1]$, i.e., $W$ is
  identified with $W^\sigma(x,y) := W(\sigma(x),\sigma(y))$ where
  $\sigma \colon [0,1] \to [0,1]$ is measure-preserving.}. The
separation of these two cases can be illustrated via a phase diagram,
as in Figure~\ref{fig:K3}, by plotting the phases in the $(p,q)$-plane
according to
the behavior of $\UT_p(H,q)$.

The constant graphon $W \equiv q$ is the
limit of random graphs $\cG(n,q)$ as $n \to \infty$, so if it were the
unique minimizer of $\UT_p(H, q)$ then $\cG(n,p)$, conditioned on
having $H$-density at least $q^{e(H)}$, approaches the typical
$\cG(n,q)$ in cut distance; this is not the case when $W \equiv q$ is
not a minimizer. Borrowing language from statistical physics,
informally, when $W \equiv q$ is a minimizer we say that there is
\emph{replica symmetry}\footnote{There is a subtle issue of uniqueness
  of the minimizer. When the constant graphon $W \equiv q$ is the
  unique minimizer, $\cG(n,q)$ represents the most likely model for
  the conditioned random graph (in terms of cut metric). However, it
  may be the case that $W \equiv q$ is a non-unique minimizer (which
  provably does not happen for $\UT_p(K_3, q)$ but I suspect that it
  does happen
  for the corresponding lower tail problem $\LT_p(K_3, q)$). When
  there are multiple distinct minimizers to the variational problem,
  all minimizers give rise to the same exponential rate, but one
  minimizer might still dominate by a lower order $\exp(o(n^2))$
  factor, which I do not know how to discern purely from the
  variational problem.}, and otherwise there is \emph{symmetry
  breaking}.

In a previous paper with Lubetzky~\cite{LZ1}, we completely identified
the upper tail replica symmetric phase whenever
$H$ is a $d$-regular graph. The phase diagram depends only on $d$. The
diagram for $H = K_3$ is shown in Figure~\ref{fig:K3} in the upper
portion (i.e., $q > p$) of the diagram\footnote{The boundary curve for
  the upper tail phase diagram for $K_3$ is
  given by the equation $( 1 + (q^{-1} - 1)^{1/(1-2q)})p = 1$.}. The lower portion of the
diagram illustrates new results in paper concerning the
lower tail problem.

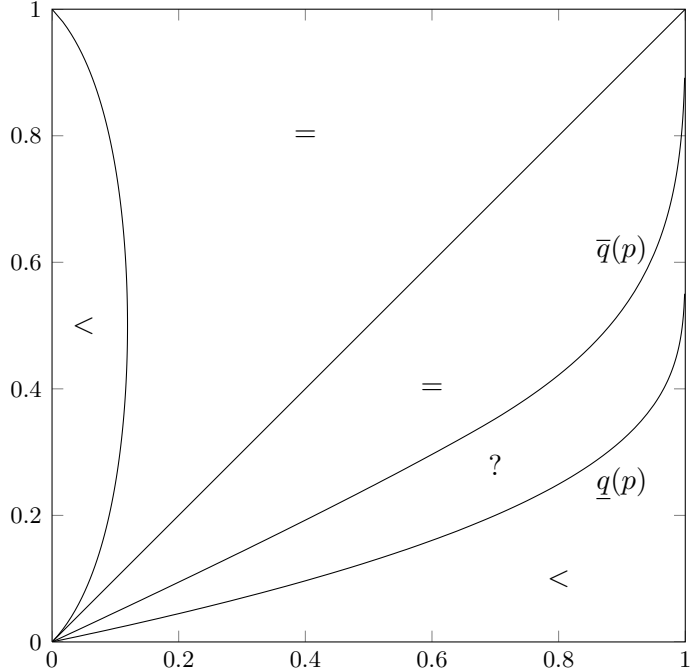
\begin{figure}
\begin{tikzpicture}
  \begin{axis}[
    xmin=0,xmax=1,ymin=0,ymax=1,
    height=10cm,width=10cm,
    ticklabel style={font=\footnotesize}
    ]
    \addplot[] table {plot-uppertailcurve.dat};
    \addplot[] table {plot-lowersymcurve.dat};
    \addplot[] table {plot-lowerbrkcurve.dat};
    \addplot[] coordinates {(0,0) (1,1)};
    \node[] at (axis cs:.4,.8) {$=$};
    \node[] at (axis cs:.05,.5) {$<$};
    \node[] at (axis cs:.6,.4) {$=$};
    \node[] at (axis cs:.7,.28) {$?$};
    \node[] at (axis cs:.8,.1) {$<$};
    \node[] at (axis cs:.9,.62) {$\overline{q}(p)$};
    \node[] at (axis cs:.9,.25) {$\underline{q}(p)$};
  \end{axis}
\end{tikzpicture}
  \caption{The phase diagram for triangle density upper tail
    variational problem
    $\UT_p(K_3, q)$ (when $q > p$) and lower tail variational
    problem $\LT_p(K_3, q)$ (when $q < p$).
    In regions marked ``$=$'', the constant graphon $W \equiv q$ is
    the unique minimizer to the variational problem. In regions marked
    ``$<$'', the constant graphon does not minimize the variational
    problem. The region marked ``?''\ is unresolved. The boundary curves $\oq$ and $\uq$ are from Theorem~\ref{thm:LT-K3}.
  }
  \label{fig:K3}
\end{figure}

\medskip

In this paper we study the corresponding lower tail variational
problem. For $0 \le q \le p \le 1$, let
\begin{equation} \label{eq:LTp}
  \LT_p(H, q) :=
  \left\{
  \begin{array}{l}
    \text{minimize } \EE[I_p(W)] \\
    \text{subject to } t(H, W) \leq q^{e(H)}.
  \end{array}\right.
\end{equation}
The connections between the large deviation problem and the variational
problem discussed earlier hold for the lower tail just as they do for the upper
tail. For example, as in~\eqref{eq:ldvar}, for fixed $0 \le
q \le p \le 1$, we have
\begin{equation} \label{eq:ldvar-LT}
\log \PP( t(H, \cG(n,p)) \le q^{e(H)} ) =
- (1+o(1)) \frac{n^2}{2}\LT_p(H, q).
\end{equation}

As observed in \cite{LZ1}, if $H$ is a bipartite graph satisfying
Sidorenko's conjecture~\cite{Sid93}, which asserts that $t(H, W) \geq
\EE[W]^{e(H)}$ for all graphons $W$, then from the constraint of
$\LT_p(H, q)$ we deduce $\EE[W] \leq q$. Since $I_p(\cdot)$ is a convex
function, from $\EE[W] \le q$ it follows that $W \equiv q$ is the
unique minimizer of $\LT_p(H, q)$. Sidorenko's conjecture remains open\footnote{The first
  unsettled case of Sidorenko's conjecture is for the graph $H$ being $K_{5,5}$ with a Hamiltonian cycle
  removed (this $H$ is sometimes called a ``M\"obius strip''). There is
  some sentiment in the community that Sidorenko's conjecture may be false
  for this graph.}, though it has been proved for certain families of
bipartite graphs $H$ such as trees, cycles, hypercubes, and bipartite
graphs containing one vertex adjacent to all vertices on the opposite
side \cite{CFS10,Hat10,KLL,LS,Sz14}. Even if Sidorenko's conjecture
were false, it could still be true that $W \equiv q$ minimizes
$\LT_p(H, q)$ for every bipartite $H$.

For the first non-bipartite case, namely $K_3$, new results in this
paper partially characterize the lower tail phase diagram, as depicted
in Figure~\ref{fig:K3}. The region marked ``?''\ remains
unresolved. For other non-bipartite graph $H$, it is possible to draw
similar partially identified phase diagrams using techniques in this
paper. We will pay special attention to the slopes of the boundary
curves at the origin.


The lower tail variational problem seems to be harder than the
corresponding upper tail problem\footnote{despite that the probabilistic problem of
  determining the order of $\log \PP(t(K_3, \cG(n,p)))$ when $p = p(n)
  \to 0$ came much later~\cite{Cha12,DK12} compared to the
  corresponding lower tail result~\cite{J90,JLR}.}. By analogy, for the classical extremal graph theory
problem of determining the range of possible triangle densities in a
graph of fixed edge density, the maximization problem (analogous to
upper tail) follows as a corollary of the classic
Kruskal--Katona theorem \cite{Kru63,Kat68}\footnote{The proof of the triangle upper tail result in
\cite{LZ1} actually uses some form of strengthening of the
Kruskal-Katona result, as we explain in Section~\ref{sec:UT}.},
whereas the corresponding minimization problem (analogous to lower tail)
was solved only relatively recently by
Razborov~\cite{Raz08} using his flag algebra machinery (also later solved
for $K_4$ by Nikiforov~\cite{Nik11} and all $K_t$ by
Reiher~\cite{Rei}). Furthermore, the qualitative nature of the phase
transition seems to be different for the upper tail and the lower
tail. It seems likely that the optimizing graphon $W$ changes continuously
as $(p,q)$ crosses upper tail phase boundary, while discontinuously for
the lower tail.

The sparse setting, with $p = p_n \to 0$ and $q/p$ kept constant, is
more difficult.  Using powerful new methods, Chatterjee and
Dembo~\cite{CD} showed that the large deviation problem in sparse
random graphs also reduces to the natural variational
problem\footnote{Some minor modifications needs to made to the
  formulation variational problem $\LT_p(H, q)$ in order to match the
  statements in \cite{CD}, namely that we only consider grpahons that
  correspond to weighted graphs on $n$ vertices. This difference is
  minor and does not affect the rest of this paper.}, provided that $p
\geq n^{-\alpha_H}$ for some explicit $\alpha_H > 0$. A similar
conclusion can be made about the lower tail variational problem using
their techniques. With Lubetzky~\cite{LZ2} we obtained the following
asymptotic solution to the corresponding variational problem: for
every fixed $\delta > 0$,
\begin{equation} \label{eq:UT-sparse}
\lim_{p\to 0} \frac{\UT_p(K_3, (1+\delta)^{1/3} p)}{p^2
  \log(1/p)} = \min\left\{\delta^{2/3}, \frac{2}{3} \delta\right\},
\end{equation}
and as a corollary, as long as $p = p_n \to 0$ with
$p_n \geq n^{-1/42} \log n$, we have
\[
\PP( t(K_3, \cG(n,p))\geq (1+\delta)p^3) = \exp\left( -(1 -
  o(1))\min\left\{\frac{\delta^{2/3}}{2},
    \frac{\delta}{3}\right\}n^2p^2\log\frac{1}{p} \right).
\]

In this paper, we also study the lower tail variational problem as $p
\to 0$. A nice feature of the lower tail problem in the sparse limit is
that instead of being concerned with the entire phase boundary curve,
we can focus on its slope at the origin.\footnote{For the
  upper tail boundary curve, the slope at the origin is always 1.}


The lower tail problem was also recently analyzed by Janson and Warnke
\cite{JW} from a completely different perspective (not relating to the
variational problem). In the triangle case, for $n^{-1/2} \ll p \to
0$, they were able to
determine the large deviation rate of
$\PP(t(K_3, \cG(n,p)) \le (1-\delta) p^3)$ for the two extremes
$\delta = o(1)$ and $\delta = 1 - o(1)$. They left as an open
question what happens for fixed $\delta \in (0,1)$, which is the subject
of this paper.

There are other variants of the variational problem being studied in
literature. For exponential random graphs, see
\cite{AR13,KY,CD11,LZ1,YRF,RY13,Yin13,Zhu14,AZ14a,YZ14a,YZ14b,Yin14}.
For the variational problem where several subgraph densities are
simultaneously constrained (e.g., edge and triangle densities both
fixed), see \cite{RS,RS13,KRRS,RRS14,AZ14b}.

\medskip

Section~\ref{sec:results} contains statements of the results.
Section~\ref{sec:UT} reviews the techniques used in proof of the
upper tail results from \cite{LZ1}. Section~\ref{sec:LT-K3} concerns
the upper tail problem for triangle densities. Section~\ref{sec:LT-H}
concerns general $H$-densities. The methods in
Sections~\ref{sec:LT-K3} and \ref{sec:LT-H} are different since the
 techniques for triangles seem to be quantitatively superior but
do not extend to all graphs.
Section~\ref{sec:end} concludes with some open problems.

\section{Results} \label{sec:results}

\subsection{Triangle density}

Here is our main result concerning the lower tail variational problem
$\LT_p(K_3,q)$ for triangle densities. See Figure~\ref{fig:K3}.

\begin{theorem}
  \label{thm:LT-K3}
  There exist functions $\uq, \oq \colon (0,1)
  \to (0,1)$ satisfying $0 < \uq(p) \leq \oq(p)
  \leq p$ for $0 < p < 1$ with the following properties. Whenever $\oq(p) < q < p$, the
  constant graphon $W \equiv q$ is the unique minimizer for
  $\LT_p(K_3, q)$. Whenever $0 < q < \uq(p)$, the
  constant graphon $W \equiv q$ does not minimize
  $\LT_p(K_3, q)$. Furthermore, $\lim_{p \to 0} \uq(p)/p = 0.209\dots$
  while $\lim_{p \to 0} \oq(p)/p = 0.466\dots$.
\end{theorem}

The two curves $\uq(p)$ and $\oq(p)$ are drawn in
Figure~\ref{fig:K3}. The nature of $\LT_p(K_3, q)$ remains unresolved
for $(p,q)$ between these two curves.

In Theorem~\ref{thm:LT-K3} and elsewhere, $0.466\dots$ denotes the
unique $0 < r < 1$ satisfying $\frac 32 r \log r - r + 1 =0$, and
$0.209\dots$ is defined as the maximum
value of $r < 1$ such that that the function $f_r(x)$ in
\eqref{eq:K3-h-f-BIP} (also see Figure~\ref{fig:K3-h-f-BIP}) has a zero in the open interval $(0, r)$.

\subsection{General subgraph density}

We extend Theorem~\ref{thm:LT-K3} to general subgraph
counts. No serious effort is made here at optimizing the
quantitative bounds.

\begin{theorem}
  \label{thm:LT-H}
  Let $H$ be a graph. There exists a function $\oq \colon (0,1) \to
  (0,1)$ with $\lim_{p \to 0} \oq(p)/p < 1$ such that whenever $\oq(p) < q \leq p$, the constant graphon $W
  \equiv q$ is the unique minimizer for $\LT_p(H, q)$.

  Furthermore, if $H$ is not bipartite, then there exists a function $\uq \colon (0,1)
  \to (0,1)$ with $\lim_{p \to 0} \uq(p)/p > 0$ such that whenever $0 \leq q < \uq(p)$, the
  constant graphon $W \equiv q$ does not minimize
  $\LT_p(K_3, q)$.
\end{theorem}

The proof of the triangle case, Theorem~\ref{thm:LT-K3}, makes use of Goodman's inequality \cite{Goo59}:
\[
t(K_3, W) + t(K_3, 1 - W) \geq 1/4.
\]
If $H$ satisfies $t(H, W) + t(H, 1-W) \geq 2^{-e(H)+1}$ for all
graphons $W$ (such a graph $H$ is sometimes called ``common'' in the
context of Ramsey multiplicities), then the same method can be used to
establish regions where $W \equiv q$ is a minimizer of $\LT_p(H, q)$
(though the actual regions will not be the same as in
Figure~\ref{fig:K3} due to other technical reasons). However, $t(H, W)
+ t(H, 1-W) \geq 2^{-e(H)+1}$ does not hold in general. For example, Thomason~\cite{Tho89} showed that $K_t$ is a counterexample
for all $t\ge 4$. Consequently, the proof method of
Theorem~\ref{thm:LT-K3} does not seem to extend to all $H$.
Theorem~\ref{thm:LT-H} for general $H$ is proved using a different
method, which seems quantitatively inferior to the method for
triangles.

For bipartite $H$, I conjecture that there is no phase transition:

\begin{conjecture} \label{conj:LT-bipartite}
  Let $H$ be a bipartite graph. Then the constant function $W \equiv
  q$ is always the unique minimizer of $\LT_p(H, q)$.
\end{conjecture}

As mentioned in the introduction, the conjecture holds for any $H$
for which Sidorenko's conjecture is true, i.e., $t(H, W) \geq
\EE[W]^{e(H)}$ for all graphons
$W$. Conjecture~\ref{conj:LT-bipartite} may be true even if
Sidorenko's conjecture were not.

\subsection{Sparse limit}

Since $I_p$ is decreasing in $[0,p]$ and increasing in
$[p,1]$, any minimizing $W$ for $\LT_p(H, q)$ satisfies $0 \leq W
\leq p$ almost everywhere. Let
\[
h(x) := x \log x - x + 1
\]
so that
\[
\lim_{p\to 0} p^{-1} I_p(px) = h(x)
\]
uniformly for $x \in
[0,1]$. It follows that for every graph $H$ and $0 \leq r \leq 1$ we have
\begin{equation} \label{eq:limLT}
\lim_{p \to 0} p^{-1} \LT_p(H, pr) = \LT(H, r)
\end{equation}
where
\begin{equation} \label{eq:LT}
  \LT(H, r) :=
  \left\{
  \begin{array}{l}
    \text{minimize } \EE[h(W)] \\
    \text{subject to } t(H, W) \leq r^{e(H)}.
  \end{array}\right.
\end{equation}
It would be interesting to solve this variational problem. As before, a basic
question is whether the constant function $W
\equiv r$ is a minimizer. Here is the main conjecture.

\begin{conjecture} \label{conj:h} Let $H$ be a non-bipartite graph and
  $0 \le r \le 1$. There exists a
  $0<r_{H}^*<1$ so that $W \equiv r$ minimizes $\LT(H, r)$ if and only
  if $r \geq r_H^*$. Furthermore, $W$ is the unique minimizer for
  $\LT(H, r)$ if and only if $r > r_H^*$.
\end{conjecture}

The conjecture remains open for any non-bipartite graph $H$.
For the bipartite case:

\begin{conjecture}
  \label{cor:h-bip}
  The constant graphon $W \equiv r$ is the unique
  minimizer for $\LT(H, r)$ for every bipartite graph $H$ and every $0 \leq r
  \leq 1$.
\end{conjecture}

In proving Theorem~\ref{thm:LT-K3} and Theorem~\ref{thm:LT-H}, we
obtain the following results in the direction of the above
conjectures.

\begin{theorem} \label{thm:LT-h-K3}
  If $0.466\dots < r \leq 1$, then the constant graphon $W \equiv r$
  uniquely minimizes $\LT(K_3, r)$. If $0 \le r < 0.209\dots$, then
  the constant graphon $W \equiv r$ does not minimize $\LT(K_3, r)$.
\end{theorem}

I conjecture that, in Conjecture~\ref{conj:h}, $r_{K_{3}}^* = 0.209\dots$.

\begin{theorem} \label{thm:LT-h-H}
  Let $H$ be a graph. There exists $\ovr_H < 1$ such that $W \equiv r$
  uniquely minimizes $\LT(H, r)$ whenever $\ovr_H \leq r \leq 1$. If
  $H$ is nonbipartite, then there exists $\ur_H > 0$ such that $W
  \equiv r$ does not minimize $\LT(H, r)$ for $0 \leq r < \ur_H$.
\end{theorem}


Combining these results with the framework of
Chatterjee and Dembo \cite{CD}, we obtain

\begin{corollary} \label{cor:LT}
  Let $H$ be a graph. There is some
  explicit $\alpha_H > 0$ so that for $p = p_n \to 0$ with $p \ge
  n^{-\alpha}$, the following large deviation results hold.

  There exists $\ovr_H < 1$ so that for all $r \in (\ovr_H, 1)$,
  \[
  \lim_{n \to \infty} \frac{2}{n^2 p} \log \PP(t(H, \cG(n,p)) \le
  (rp)^{e(H)}) = - h(r).
  \]
  If $H$ is non-bipartite, then there exists $\ur_H > 0$ so that for
  all $r \in (0,\ur_H)$,
  \[
  \liminf_{n \to \infty} \frac{2}{n^2 p} \log \PP(t(H, \cG(n,p)) \le
  (rp)^{e(H)}) > - h(r).
  \]
  For $H = K_3$, we may take $\ovr_{K_3} = 0.466\dots$ and $\ur_{K_3}
  = 0.209\dots$.
\end{corollary}

\section{Review of the proof for triangle upper tails} \label{sec:UT}

We begin with a quick review of the proof of the upper tail result
from \cite{LZ1}, as some of the ideas are used in the proof of Theorem~\ref{thm:LT-K3}.

The following extension of H\"older's
inequality is very useful. See \cite[Corollary 3.2]{LZ1}.

\begin{proposition} \label{prop:holder}
  Let $H$ be a graph with maximum degree $\Delta$. For any symmetric measurable function $W \colon [0,1]^2 \to \RR$,
   we have $t(H, W) \leq \EE[|W|^\Delta]^{e(H)/\Delta}$. In particular,
   $t(K_3, W) \leq \EE[W^2]^{3/2}$.
\end{proposition}

The inequality can be proved via repeated applications of H\"older's
inequality (when $H = K_3$, it takes three applications of the Cauchy--Schwarz
inequality). Observe that the inequality $t(K_3, W) \leq \EE[W^2]^{3/2}$
strengthens a corollary of the Kruskal--Katona
theorem on the maximum possible triangle density in a
graph of given edge density: $t(K_3, W) \le \EE[W]^{3/2}$.

The following result from \cite{LZ1} gives the full replica
symmetric phase for $\UT_p(K_3,q)$, the upper tail problem for
triangle densities.

\begin{theorem}
  \label{thm:UT-K3}
  Let $0 < p \le q < 1$. If the point $(q^2, I_p(q))$ lies on the convex
  minorant of the function $x \mapsto I_p(\sqrt{x})$, then $W \equiv
  q$ is the unique minimizer of $\UT_p(K_3,q)$.
\end{theorem}

The upper tail boundary curve in Figure~\ref{fig:K3} is characterized
by the condition in Theorem~\ref{thm:UT-K3}. See \cite[Lemma 3.1]{LZ1} for
the proof of symmetry breaking, i.e., $\UT_p(K_3,q) < I_p(q)$, to the
left of the boundary curve.

\begin{proof}
  By the convex minorant condition, the tangent line to the function $I_p(\sqrt{x})$ at $x = q^2$
  lies below the function, so that
  \[
  I_p(\sqrt{x}) \geq I_p(q) + \frac{I'_p(q)}{2q}(x - q^2), \quad
  \forall x \in [0,1].
  \]
  Replacing $x$ by $x^2$, we get
  \begin{equation} \label{eq:UT-K3-conv-supp}
  I_p(x) \geq I_p(q) + \frac{I'_p(q)}{2q}(x^2 - q^2), \quad
  \forall x \in [0,1].
\end{equation}
Note that $I'_p(q) >0$ since $q > p$.

  Suppose graphon $W$ satisfies $t(K_3, W) \geq q^3$.
  By Proposition~\ref{prop:holder}, we have $\EE[W^2] \geq t(K_3,
  W)^{3/2} \geq q^2$. Thus \eqref{eq:UT-K3-conv-supp} implies that
  \[
  \EE[I_p(W)] \geq I_p(q) + \frac{I'_p(q)}{2q} (\EE[W^2] - q^2) \geq I_p(q).
  \]
  This shows that $W \equiv q$ is a minimizer for $\UT_p(K_3, q)$, and
  furthermore it is not too hard to check equality conditions to
  verify that this is the unique minimizer.
\end{proof}

\section{Triangle lower tails} \label{sec:LT-K3}

In this section we prove Theorems~\ref{thm:LT-K3} and \ref{thm:LT-h-K3}.

\subsection{Replica symmetry phase}

We begin with a small modification of Goodman's theorem~\cite{Goo59}
(which is usually generally stated for $U + W \equiv 1$).

\begin{lemma}
  \label{lem:goodman}
  If $U$ and $W$ are graphon such that $U + W \geq 2q$ for some
  constant $q \geq 0$, then
  \[
  t(K_3, W) + t(K_3, U) \geq 2q^3.
  \]
\end{lemma}

\begin{proof}
  By decreasing $U$ and $W$ (while remaining nonnegative), we may assume that they are
  graphons satisfying $U + W \equiv 2q$. Let $U = q + X$ and $W = q -
  X$ for some symmetric measurable function $X \colon [0,1]^2 \to
  \RR$. Then
  \begin{align*}
  t(K_3, W) + t(K_3, U)
  &= t(K_3, q + X) + t(K_3, q - X)
  \\&= 2q^3 + 6q \, t(K_{1,2}, X)
  \\&\geq 2q^3 + 6q (\EE[X])^2
  \geq 2q^3. \qedhere
\end{align*}
\end{proof}

For any $a \in \RR$ we write $a_+ := \max\{a, 0\}$. In the Proposition
below, $a_+^2$ means $(a_+)^2$. The inequality \eqref{eq:K3-hyp} below
is motivated by considering the tangent line to $x \mapsto I_p(2q -
\sqrt{x})$ at $x = q^2$, as in the proof of Theorem~\ref{thm:UT-K3}.

\begin{proposition} \label{prop:K3=}
  Let $0 < q \leq p < 1$ be such that
  \begin{equation} \label{eq:K3-hyp}
  I_p(x) \geq I_p(q) + \frac{-I_p'(q)}{2q} ((2q-x)^2_+ - q^2) \quad
  \forall x \in [0,p].
  \end{equation}
  Then $W \equiv q$ is the unique minimizer of $\LT_p(K_3, q)$.
\end{proposition}

\begin{proof}
  Suppose $W$ satisfies $t(K_3, W) \leq q^3$.
  Apply Lemma~\ref{lem:goodman} to $W$ and $U := (2q - W)_+$ to
  obtain
  \[
  t(K_3, (2q - W)_+) \ge 2q^3 - t(K_3, W) \ge q^3.
  \]
  Next, apply Proposition~\ref{prop:holder} and we obtain
  \[
  \EE[(2q-W)_+^2] \ge t(K_3, (2q-W)_+)^{2/3} \ge q^2.
  \]
  By \eqref{eq:K3-hyp} we have (note that $I_p'(q) \le 0$ as $q \leq p$)
  \[
  \EE[I_p(W)] \ge I_p(q) + \frac{-I_p'(q)}{2q} ( \EE[(2q-W)_+^2] -
  q^2)
  \geq I_p(q).
  \]
  It follows that $\LT_p(K_3, q) = I_p(q)$. To show that $W \equiv q$
  is the unique minimizer, observe that in order for any other $W$ to
  be a minimizer, equality must occur at every step above. In
  particular, if \eqref{eq:K3-hyp} has single point of equality,
  namely for $x = q$, then the uniqueness of $W$ is clear. Otherwise,
  one can check (details omitted, but see Figure~\ref{fig:K3-sym-f})
  that that \eqref{eq:K3-hyp} has at most two points of equality, with
  one being $x = q$, so that if $W$ has any positive mass with value
  being the other point of equality, then it would be impossible for
  $t(K_3, W) = q^3$ to hold. This shows that $W \equiv q$ is the
  unique minimizer.
\end{proof}

To derive results about the phase diagram, we shall invoke various technical statements (referred to as ``Facts'') about the
functions $I_p$ and $h$. Using Fact~\ref{fact:K3-hyp-ineq} below we obtain the $0 \leq p \leq 1/2$
portion of the curve $\oq$ of Theorem~\ref{thm:LT-K3}, which is given
by the implicit equation $I_p(q) + \frac12 q I_p'(q) =0$, and shown
in Figure~\ref{fig:K3}. The rest of the curve in Figure~\ref{fig:K3}
is produced by numerically checking \eqref{eq:UT-K3-conv-supp}. Taking
the $p\to 0$ limit of the implicit equation, we see that the slope at
the origin equals to $\ovr = 0.466\dots$, where $\ovr$ satisfies $h(\ovr)
+ \frac12 \ovr h'(\ovr) = 0$. This completes the proof of the replica
symmetric phase in Theorem~\ref{thm:LT-K3}.

\begin{fact} \label{fact:K3-hyp-ineq}
  For $0 < q \leq p \leq 1/2$, \eqref{eq:K3-hyp} holds for all $x \in
  [0,p]$ if and only if
  it holds at $x = p$.
\end{fact}

\begin{proof}
  Let
  \begin{equation}
    \label{eq:K3-sym-f}
    f(x) := f_{p,q} := I_p(x) - I_p(q) + \frac{I'_p(q)}{2q}((2q-x)_+^2 -q^2).
  \end{equation}
  We plotted $f$ for some representative values of $(p,q)$ in
  Figure~\ref{fig:K3-sym-f}.
  \begin{figure}
  \pgfplotsset{
    every axis/.append style={
      width=4cm,
      xtick={0,.05,.1},
      minor x tick num = 4,
      axis x line=middle,
      axis y line=left,
      scaled ticks=false,
      ymin=-.002,ymax=.01,
      xmin=0,xmax=.1,
      xticklabel style={font=\tiny,/pgf/number format/.cd,fixed,precision=2},
      yticklabel style={font=\tiny,/pgf/number format/.cd,fixed,fixed
        zerofill,precision=3}
    },
    every axis plot post/.append style={black, mark=none}
    }
  \begin{tikzpicture}
    \begin{axis}
      \addplot table {plot-K3sym045.dat};
    \end{axis}
    \node at (1.3,-.4) {\footnotesize $(p,q) = (0.1,0.045)$};
    \begin{axis}[xshift=3cm,yticklabels={,,}]
      \addplot[] table {plot-K3sym047.dat};
    \end{axis}
    \node at (4.3,-.4) {\footnotesize $(p,q)=(0.1,0.047)$};
    \begin{axis}[xshift=6cm,yticklabels={,,}]
       \addplot[] table {plot-K3sym05.dat};
     \end{axis}
     \node at (7.3,-.4) {\footnotesize $(p,q)= (0.1,0.05)$};
     \begin{axis}[xshift=9cm,yticklabels={,,}]
       \addplot[] table {plot-K3sym06.dat};
     \end{axis}
     \node at (10.3,-.4) {\footnotesize $(p,q)=(0.1,0.06)$};
\end{tikzpicture}

  \caption{Plots of $f_{p,q}$ from (\ref{eq:K3-sym-f}) for $p=0.1$ and
    various values of $q$.}
  \label{fig:K3-sym-f}
\end{figure}
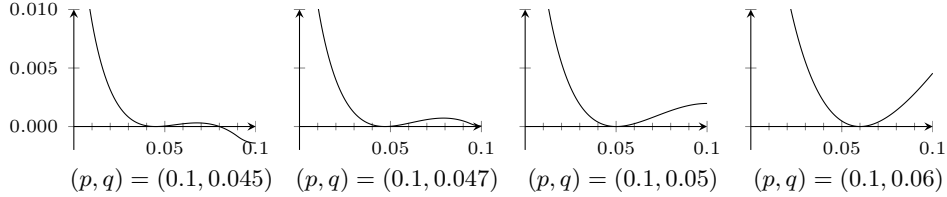

  Suppose $f(p) \geq 0$. We have
  \[
  f'(x) = I'_p(x) - \frac{I'_p(q)}{q}(2q - x)_+
  \]
  and
  \[
  f''(x) = I''_p(x) + \frac{I'_p(q)}{q} 1_{x < 2q} = \frac{1}{x(1-x)}
  + \frac{I'_p(q)}{q} 1_{x < 2q}.
  \]
  Since $p \leq 1/2$, $f''(x)$ is decreasing for $0 <x < \min\{p,
  2q\}$. Clearly $f''$ is positive near $x = 0$. We consider two
  cases.

  Case I: $f''(x) > 0$ for all $0 < x < \min\{p, 2q\}$. Then $f$ is
  convex on $(0, \min\{p, 2q\})$. We know that $f(q) = f'(q) =
  0$. Then $f(x) \geq 0$ for all $x \in [0, \min\{p,2q\}]$. If $2q
  \geq p$, then we are done, otherwise, note that $f(x) = I_p(x) -
  I_p(q) - qI'_p(q)/2$ for $x \in [2q, p]$, and it is decreasing on
  this interval. Since we assumed that $f(p) \geq 0$, we obtain $f(x)
  \geq 0$ for all $x \in [0,p]$.

  Case II: there is some $x_0 \in (0,\min\{p, 2q\})$ such that
  $f''(x_0) = 0$. So $f$ is convex on $(0,x_0)$ and concave on $(x_0,
  \min\{p, 2q\})$. We assumed that $f(p) \geq 0$, so $f(\min\{p, 2q\})
  \ge 0$ since if $2q < p$ then $f$ is decreasing on $(2q, p)$. Since
  $f(q) = f'(q) = 0$, an analysis of the convexity of $f$ shows that
  it is nonnegative on $[0,p]$.
\end{proof}

For the sparse limit $p \to 0$, the proof of the first half of
Theorem~\ref{thm:LT-h-K3} is nearly identical. It follows from the
following two propositions, whose proofs we omit.

\begin{proposition}
    Let $0 \leq r \leq 1$ be such that
    \begin{equation}\label{eq:K3-h-hyp}
    h(x) \geq h(r) + \frac{-h'(r)}{2r}( ( 2r - x)_+^2 -r^2) \quad
    \forall x \in [0,1].
  \end{equation}
  Then $W \equiv r$ is the unique minimizer of $\LT(K_3, r)$.
\end{proposition}

\begin{fact}
  The inequality \eqref{eq:K3-h-hyp} holds for all $x \in [0,1]$ if
  and only if holds for $x = 1$, which holds if and only if $r \geq
  \ovr = 0.466\dots$, where $\ovr$ satisfies $h(\ovr) + \frac12 \ovr h'(\ovr) = 0$.
\end{fact}

\subsection{Symmetry breaking phase} \label{sec:K3-brk}

Now we explain the lower curve $\uq$ in Figure~\ref{fig:K3}.
It is obtained by by restricting the variational problem
$\LT_p(K_3, q)$ to graphons $W$ of the form $\BIP_{a,b}$, where
$\BIP_{a,b}$, for $0\leq a, b \leq 1$, is defined by
\begin{equation} \label{eq:BIP}
\BIP_{a,b}(x,y) := \begin{cases} a & \text{if } (x,y) \in [0,1/2]^2
    \cup (1/2,1]^2, \\
    b & \text{if } (x,y) \in [0,1/2] \x (1/2,1] \cup (1/2,1] \x
    [0,1/2].
  \end{cases}
\end{equation}
There is symmetry breaking as long as we can find $0 \le a, b \le p$
satisfying
\begin{equation} \label{eq:BIP-K3-goal}
\EE[I_p(\BIP_{a,b})] = \frac12 I_p(a) + \frac12 I_p(b) < I_p(q)
\end{equation}
and
\begin{equation}\label{eq:BIP-K3-constr}
t(K_3,\BIP_{a,b}) =
\frac14 a^3 + \frac34 ab^2 \le q^3.
\end{equation}
We can assume that $0
\leq a \leq q \leq b \leq p$, since otherwise swapping $a$ and $b$
reduces $t(K_3, W)$ (observe that $t(K_3, \BIP_{a,b}) - t(K_3,
\BIP_{b,a}) = \frac14 (a-b)^3$) while keeping $\EE[I_p(W)]$
constant.

Set $b = \sqrt{(4q^3-a^3)/(3a)}$ so that $t(K_3,\BIP_{a,b}) =
q^3$. There is symmetry breaking if
\begin{equation} \label{eq:K3-f-BIP}
f(x)  := f_{p,q} (x) := \frac{1}{2}I_p(x) + \frac12
I_p\left(\sqrt{\frac{4q^3-x^3}{3x}} \right) - I_p(q)
\end{equation}
is negative for some $0 \leq x \leq q$, where $f$ is only defined for
$\sqrt{(4q^3-x^3)/(3x)} \leq p$. Some representative examples of $f$
are plotted in Figure~\ref{fig:K3-f-BIP}. For every $p$, and
sufficiently small $q$, $f(x)$ becomes negative in a region away from
$x = q$.


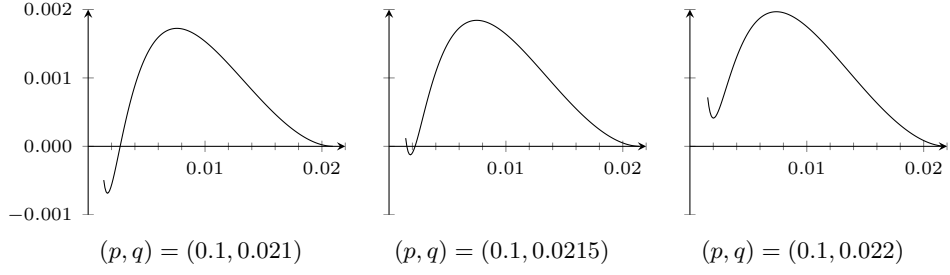
\begin{figure}
  \pgfplotsset{
    every axis/.append style={
      width=5cm,
      axis x line=middle,
      axis y line=left,
      scaled ticks=false,
      xtick={0,.01,.02},
      minor x tick num = 4,
      ymin=-.001,ymax=.002,
      xmin=0,xmax=.022,
      xticklabel style={font=\tiny,/pgf/number format/.cd,fixed,precision=2},
      yticklabel style={font=\tiny,/pgf/number format/.cd,fixed,fixed
        zerofill,precision=3}
    },
    every axis plot post/.append style={black, mark=none}
    }
  \begin{tikzpicture}
    \begin{axis}
      \addplot table {plot-K3brk021.dat};
    \end{axis}
    \node at (1.5,-.5) {\footnotesize $(p,q) = (0.1,0.021)$};
    \begin{axis}[xshift=4cm,yticklabels={,,}]
      \addplot[] table {plot-K3brk0215.dat};
    \end{axis}
    \node at (5.5,-.5) {\footnotesize $(p,q) = (0.1,0.0215)$};
    \begin{axis}[xshift=8cm,yticklabels={,,}]
       \addplot[] table {plot-K3brk022.dat};
     \end{axis}
     \node at (9.5,-.5) {\footnotesize $(p,q) = (0.1,0.022)$};
\end{tikzpicture}
  \caption{The plot of $f_{p,q}$ from (\ref{eq:K3-f-BIP}) for $p=0.1$
    and various values of $q$.}
    \label{fig:K3-f-BIP}
\end{figure}

Now we prove the claims in Theorem~\ref{thm:LT-K3} more
rigorously. For every $p > 0$, if $q$ is sufficiently small so that
$\frac12 I_p(0) < I_p(q)$, then $W = \BIP_{0,p}$ satisfies $t(K_3,W) = 0$
while $\EE[I_p(W)] = \frac12 I_p(0) < I_p(q)$, so that $\LT_p(K_3, q)
< I_p(q)$.

The argument in the previous paragraph does not give the optimal $\uq$
in Theorem~\ref{thm:LT-K3}. To prove that
$\uq$ can be chosen so that $\lim_{p\to 0} \uq(p)/p = 0.209\dots$,
it suffices, by \eqref{eq:limLT}, to prove the second half of
Theorem~\ref{thm:LT-h-K3}, that $\LT(K_3, r) < h(r)$ for all $r < r_1
= 0.209\dots$. As before, we seek $0 \le a \le r \le b \le 1$ with
\[
\frac12 h(a) + \frac12 h(b) < h(r)
\]
and
\[
\frac14 a^3 + \frac34 ab^2 \leq r^3.
\]
Let
\begin{equation}
    \label{eq:K3-h-f-BIP}
  f(x) := f_{r} (x) := \frac{1}{2}h(x) + \frac12
  h\left(\sqrt{\frac{4r^3-x^3}{3x}} \right) - h(r).
\end{equation}
\begin{figure}
  \pgfplotsset{
    every axis/.append style={
      width=5cm,
      axis x line=middle,
      axis y line=left,
      scaled ticks=false,
      xtick={0,.1,.2},
      minor x tick num = 4,
      ymin=-.005,ymax=.02,
      xmin=0,xmax=.22,
      xticklabel style={font=\tiny,/pgf/number format/.cd,fixed,precision=2},
      yticklabel style={font=\tiny,/pgf/number format/.cd,fixed,fixed
        zerofill,precision=3}
    },
    every axis plot post/.append style={black, mark=none}
    }
  \begin{tikzpicture}
    \begin{axis}
      \addplot table {plot-K3hbrk2.dat};
    \end{axis}
    \node at (1.5,-.5) {\footnotesize $r=0.2$};
    \begin{axis}[xshift=4cm,yticklabels={,,}]
      \addplot[] table {plot-K3hbrk209.dat};
    \end{axis}
    \node at (5.5,-.5) {\footnotesize $r=0.209$};
    \begin{axis}[xshift=8cm,yticklabels={,,}]
       \addplot[] table {plot-K3hbrk21.dat};
     \end{axis}
     \node at (9.5,-.5) {\footnotesize $r=0.21$};
\end{tikzpicture}
  \caption{The plot of $f_{r}$ from (\ref{eq:K3-h-f-BIP}) for various
    values of $r$.}
    \label{fig:K3-h-f-BIP}
\end{figure}
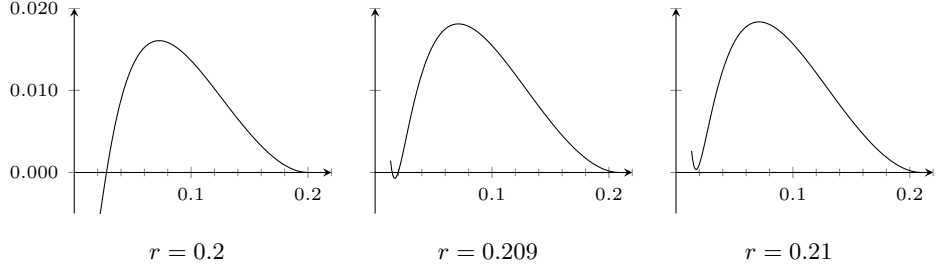

See Figure~\ref{fig:K3-h-f-BIP} for some examples of plots of $f_r$ (as before,
plotted for values of $x \leq r$ satisfying $\sqrt{(4r^3-x^3)/(3x)} \leq 1$). At the
critical $r = r_1 = 0.209\dots$, there exists $0 < a_1 <
r_1 < b_1 < 1$ such that $\frac14 a_1^3 + \frac34 a_1b_1^2 = r_1^3$ and $\frac12
h(a_1) + \frac12 h(b_1) = h(r_1)$. Now for any $0 \leq r < r_1$, let
$s = r/r_1$, so that $(a,b) = (sa_1, sb_1)$ satisfies $\frac14 a^3 + \frac34 ab^2 = r^3$. Note that
\begin{align*}
h(sx) &= sx \log (sx) - sx + 1
\\&= s (x\log x - x + 1) + (s \log s) x - s
+ 1 = sh(x) + (s \log s) x - s + 1.
\end{align*}
We have
\begin{align*}
\frac12 h(a) + \frac12 h(b) - h(r)
&= \frac12 h(sa_1) + \frac12 h(sb_1) - h(sr_1)
\\&= s (\frac12 h(a_1) + \frac12 h(b_1) - h(r_1)) + (s\log s)( \frac12
a_1 + \frac12 b_1 - r_1)
\\&< 0
\end{align*}
where we know $(a_1+b_1)/2 > r$ from
\begin{align*}
\left(\frac12 a_1 + \frac12 b_1\right)^3 - r_1^3
&= \left(\frac12 a_1 + \frac12 b_1\right)^3 - \frac14 a_1^3 - \frac34
a_1b_1^2
\\&= \left( \frac12 b_1 - \frac12 a_1\right)^3 > 0
\end{align*}
It follows that $\LT(K_3, r) < I_p(r)$ for all $0 < r < r_1 = 0.209\dots$.

\section{General subgraph lower tails} \label{sec:LT-H}

In this section we prove Theorems~\ref{thm:LT-H} and
\ref{thm:LT-h-H}. I will give the details only for
Theorem~\ref{thm:LT-h-H} concerning the sparse limit $\LT(H, r)$
as it is somewhat cleaner and contains all the ideas.
Theorem~\ref{thm:LT-H} regarding $\LT_p(H,q)$ can be proved
analogously by considering sufficiently small but fixed values of $p$.

\subsection{Replica symmetry}

For any graph $H$ and graphon $W$, we define the functional derivative $t'(H, W)$ to be the symmetric measurable function given by
\begin{equation}\label{eq:t'}
t'(H, W) = \sum_{e \in E(H)} t_e(H, W)
\end{equation}
where for each $ab \in E(H)$, we define the graphon
\begin{equation} \label{eq:t_ab}
t_{ab}(H, W)(x_a, x_b) := \int_{[0,1]^{V(H)\setminus\{a,b\}}}
\prod_{ij \in E(H) \setminus\{ab\}} W(x_i, x_j) \prod_{i \in V(H)
  \setminus\{a,b\}} dx_i.
\end{equation}
For example,
\[
t'(K_3,W)(x,y) = 3\int_{[0,1]} W(x,z)W(y,z) \, dz.
\]
For any symmetric measurable $U \colon [0,1]^2 \to [-1,1]$, and
$\delta \to 0$, we have
\[
t(H, W + \delta U) = t(H, W) + \delta \, \EE[t'(H, W) U] + O(\delta^2),
\]
which justifies calling $t'(H, W)$ the functional derivative.

\begin{lemma} \label{lem:Lagrange}
  Let $H$ be a graph and $0 < r < 1$. The variational problem
  $\LT(H, r)$ attains its minimum for some graphon $W$, and any
  such $W$ satisfies the following Lagrange multiplier condition: for
  some $\lambda \geq 0$, one has
  \[
  h'(W(x,y)) + \lambda t'(H, W)(x,y) = 0, \qquad \text{a.e.-}(x,y) \in [0,1]^2.
  \]
\end{lemma}

\begin{proof}
  That the minimum of $\LT(H, r)$ is always attained follows from the
  compactness of the space of graphons with respect to the cut
  distance and the convexity of $h$, as was already observed in
  \cite{CV11}.\footnote{We sketch here an alternate proof of the fact
    that the minimum is always attained. Let $W_n$ be a sequence of
    graphons with $t(H,W_n) \ge r^{e(H)}$ and $\EE[h(W_n)] \to
    \LT(H,r)$. By compactness of the space of
  graphons~\cite{LS07}, there exists a subsequential limit $W$ so that
  $\delta_\square(W_n, W) \to 0$ along some subsequence. Restrict to
  this subsequence. We have $t(H, W_n) \to t(H, W)$, so that $t(H,W)
  \ge r^{e(H)}$. It remains to show that $\EE[h(W)] \leq \lim
  \EE[h(W_n)] = \LT(H,r)$. We do not lose
  anything by assuming that $\norm{W_n - W}_\square \to 0$. Let
  $\cP_m$ denote the partition of the unit interval $[0,1]$ into $m$
  equal-length intervals. Let $W_{\cP_m}$ denote $W$ with its value
  inside each $I_i \x I_j$ replaced by its average, for every $I_i,
  I_j \in \cP_m$. Define $(W_n)_{\cP_m}$ similarly. Then $\norm{W_n -
    W}_\square \to 0$ implies that $(W_n)_{\cP_m} \to W_{\cP_m}$
  pointwise a.e.\ as $n \to \infty$. By convexity, we have $\lim_{n
    \to \infty} \EE[h(W_n)] \geq \liminf_{n \to \infty}
  \EE[h((W_n)_{\cP_m})] = \EE[h(W_{\cP_m})]$. Furthermore, $W_{\cP_m}
  \to W$ pointwise a.e.\ by the Lebesgue density theorem, so $\lim_{m
    \to \infty} \EE[h(W_{\cP_m})] = \EE[h(W)]$. It follows that
  $\EE[h(W)] \leq \lim \EE[h(W_n)] = \LT(H,r)$, as desired.}

  Suppose $W$ minimizes $\LT(H, r)$. We claim that for any symmetric
  measurable function $U \colon [0,1]^2 \to [-1,1]$ such that $0 \leq
  W + U \leq 1$ and $\EE[t'(H, W) U] < 0$, one has $\EE[h'(W) U] \geq
  0$. Indeed, consider the graphon $W + \delta U$ for $\delta \searrow
  0$. One has
  \[
  t(H, W + \delta U) - t(H, W) = \delta \EE[t'(H, W) U] + O(\delta^2).
  \]
  Thus $t(H, W + \delta U) < t(H, W) \leq r^{e(H)}$ for sufficiently
  small $\delta > 0$, and hence $\EE[h(W + \delta U)] \geq
  \EE[h(W)]$ since $W$ minimizes $\LT(H, r)$. On the other hand,
  \[
  \lim_{\delta \to 0} \frac{\EE[h(W + \delta U) - h(W)]}{\delta} = \EE[h'(W) U],
  \]
  so that $\EE[h'(W) U] \geq 0$ as claimed. The interchange of limit
  and expectation above can be justified by writing $U = U_+ - U_-$,
  where $U_+ = \max\{U, 0\}$ and $U_- = \max\{-U, 0\}$. Since $h$ is
  convex, $(h(W + \delta U_+) - h(W))/\delta$ is pointwise
  monotonically decreasing as $\delta \searrow 0$, and likewise $(h(W
  - \delta U_-) - h(W))/\delta$ is pointwise monotonically
  increasing. So the interchange of limit and expectation follows from
  the monotone convergence theorem.

  The lemma follows easily from the claim we just proved.
\end{proof}

\begin{lemma} \label{lem:W-lb}
  Let $H$ be a graph with $m$ edges, and $0 < r \leq 1$. Let $W$
  minimize $\LT(H, r)$. Then $W \geq r^{mr^{-m}}$ almost everywhere.
\end{lemma}

\begin{proof}
  Let $c = r^{mr^{-m}}$.
  Suppose on the contrary that $W < c$ on a set of positive
  measure. Let $\lambda$ be the Lagrange multiplier in
  Lemma~\ref{lem:Lagrange}. From \eqref{eq:t'} we have $t'(H, W) \leq m$ everywhere. By
  considering a positive-measure set of $(x,y)$ such that $W(x,y) <
  c$, we find
  \[
  0 = h'(W(x,y)) + \lambda t'(H, W)(x,y) < h'(c) + m\lambda.
  \]
  So that
  \[
  \lambda > \frac{-h'(c)}{m} = \frac{\log(1/c)}{m}.
  \]
  Therefore, up to a set of measure zero, for every $(x,y)$ with $W(x,y) \geq r^m$, we have
  \[
  t'(H, W)(x,y) = \frac{- h'(W(x,y))}{\lambda}
                < \frac{- m h'(r^m)}{\log(1/c)}
                = \frac{m \log(r^{-m})}{\log (r^{-mr^{-m}})}
                = mr^{m}.
  \]
  On the other hand, for every $(x,y)$ with $W(x,y) < r^m$, we have $t'(H,W)(x,y) W(x,y) <
  mr^m$. Thus $t'(H,W) W < m r^m$ almost everywhere. By \eqref{eq:t'} and
  \eqref{eq:t_ab}, we have
  \[
  t(H, W) = \frac{1}{m} \EE[t'(H, W) W] < r^m.
  \]
  However, any $W$ with $t(H, W) < r^m$ cannot minimize $\LT(H,
  r)$. This gives the desired contradiction.
\end{proof}

\begin{lemma} \label{lem:logW}
  If $t(H, W) \leq r^{e(H)}$, then $\EE[\log W] \leq \log r$.
\end{lemma}

\begin{proof}
  The lemma follows from Jensen's inequality:
  \begin{align*}
    m \EE[\log W] &=
  \int_{[0,1]^{V(H)}} \log \left(\prod_{ij \in E(H)} W(x_i, x_j)\right) \prod_{i
    \in V(H)} d x_i
  \\
  &\leq \log \left( \int_{[0,1]^{V(H)}} \prod_{ij \in
    E(H)} W(x_i, x_j) \prod_{i \in V(H)} dx_i \right)
  = \log t(H, W) \leq m \log r. \qedhere
\end{align*}
\end{proof}

\begin{proposition} \label{prop:LT-h-W}
  Let $H$ be a graph with $m \geq 1$ edges. Let $r = r_m$ (see Table~\ref{tab:r_m})  be the unique
  solution in the interval $(0,1)$ to the equation
  \[
  h(r^{mr^{-m}}) = h(r) + r'h(r) (\log (r^{mr^{-m}}) - \log r)
  \]
  Then $\LT(H,
  r)$ is uniquely minimized by the constant function $W \equiv r$ for
  all $r \geq r_m$.
\end{proposition}

\begin{proof}
  Let $r \geq r_m$.  Let $W$ be a minimizer for $\LT(H, r)$. By
  Lemma~\ref{lem:W-lb}, $W \geq r^{mr^{-m}}$ almost everywhere. Thus it follows by
  Fact~\ref{fct:h-exp} below (and it can be checked that $r_m \geq
  1/e$) that
  \begin{equation} \label{eq:LT-h-W-ineq}
    h(W) \geq h(r) + r \log r (\log W - \log r) \qquad \text{a.e.}
  \end{equation}
  Taking expectation of both sides and using $\EE[\log W] \leq \log r$
  from Lemma~\ref{lem:logW} (note that $\log r \leq 0$), we obtain
  $\EE[h(W)] \geq h(r)$, as desired. To see that $W \equiv r$ is
  unique, suppose $W$ is another minimizer of $\LT(H, r)$. Equality
  must hold everywhere in the argument. In particular,
  \eqref{eq:LT-h-W-ineq} must hold almost everywhere, which easily
  implies that $W \equiv r$ (for the critical case $r = r_m$,
  $W$ might also take the value $r^{mr^{-m}}$, but only on a set of
  measure zero since $\EE[h(W)] = h(r)$).
\end{proof}

\begin{table}
  \centering
  \begin{tabular}{ccccccccccc}
    \toprule
    $m$ & 3 & 4 & 5 & 6 & 7 & 8 & 9 & 10 & 20 & 100
    \\
    \midrule
    $r_m$ & 0.686 & 0.735 & 0.770 & 0.795 & 0.815 & 0.831 & 0.844 &
    0.855 & 0.911 & 0.973
    \\
    \bottomrule\\
  \end{tabular}
  \caption{Some values of $r_m$ from Proposition~\ref{prop:LT-h-W}.}
  \label{tab:r_m}
\end{table}

\begin{fact} \label{fct:h-exp}
  If
  \begin{equation} \label{eq:h-exp-tangent}
    h(x) \geq h(r) + r h'(r) ( \log x - \log r)
  \end{equation}
  holds for some $(x,r)
  = (x_0, r_0)$, with $0 \leq x_0 \leq r_0 \leq 1$ and $r_0 \in [1/e, 1]$, then it holds for all
  $(x,r) \in [x_0, 1] \x [r_0, 1]$.
\end{fact}

\begin{proof}
  The partial derivative of the RHS of \eqref{eq:h-exp-tangent} with
  respect to $r$ is $-(1+\log r) (\log r - \log x)$, which is at most
  zero as long as $x \leq r$ and $r \geq 1/e$. This shows that if
  \eqref{eq:h-exp-tangent} holds for some $(x,r) = (x_0, r_0)$ then it
  automatically holds for
  $(x,r) =(x_0, r)$ for all $r \in [r_0, 1]$.

  Let us now fix $r$. Let
  \begin{equation} \label{eq:h-exp-tangent-diff}
    f(x) := f_r(x) := h(x) - h(r) - rh'(r)(\log x - \log r).
  \end{equation}
  Some examples of $f_r$ are plotted in Figure~\ref{fig:h-exp-f}.
  We have
  \[
  f'(x) = \log x - \frac{r \log r}{x} \qquad \text{and} \qquad
  f''(x) = \frac{x + r \log r}{x^2}.
  \]
  So $f''(x) < 0$ for $x < -r\log r$ and $f''(x) > 0$ for $x > -r\log
  r$. Note also that $f(r) = f'(r) = 0$, and $-r\log r \leq r$ as long
  as $r \geq 1/e$. By
  analyzing the convexity of $f$ (see Figure~\ref{fig:h-exp-f}), we see that $f(x_0) \geq 0$ implies
  $f(x) \geq 0$ for all $x \in [x_0, 1]$.
\end{proof}

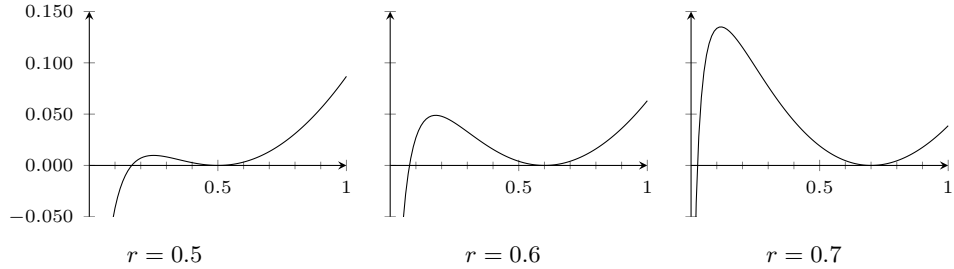
\begin{figure}
  \pgfplotsset{
    every axis/.append style={
      width=5cm,
      axis x line=middle,
      axis y line=left,
      scaled ticks=false,
      xtick={0,.5,1},
      minor x tick num = 4,
      ymin=-.05,ymax=.15,
      xmin=0,xmax=1,
      xticklabel style={font=\tiny,/pgf/number format/.cd,fixed,precision=2},
      yticklabel style={font=\tiny,/pgf/number format/.cd,fixed,fixed
        zerofill,precision=3}
    },
    every axis plot post/.append style={black, mark=none}
    }
  \begin{tikzpicture}
    \begin{axis}
      \addplot table {plot-Hsym5.dat};
    \end{axis}
    \node at (1,-.5) {\footnotesize $r=0.5$};
    \begin{axis}[xshift=4cm,yticklabels={,,}]
      \addplot[] table {plot-Hsym6.dat};
    \end{axis}
    \node at (5.5,-.5) {\footnotesize $r=0.6$};
    \begin{axis}[xshift=8cm,yticklabels={,,}]
       \addplot[] table {plot-Hsym7.dat};
     \end{axis}
     \node at (9.5,-.5) {\footnotesize $r=0.7$};
\end{tikzpicture}
  \caption{Plot of $f_r$ from (\ref{eq:h-exp-tangent-diff}) for
    various values of $r$.}
  \label{fig:h-exp-f}
\end{figure}

\subsection{Symmetry breaking}

The proof of the second part of Theorem~\ref{thm:LT-h-H} is easy. One
can fine tune the bounds as in
Section~\ref{sec:K3-brk} though we omit the analysis here.

\begin{proposition}
  Let $H$ be a nonbipartite graph. Then $\LT(H, r) < h(r)$ for all $r
  < 0.186$.
\end{proposition}

\begin{proof}
  The graphon $W = \BIP_{0,1}$ satisfies $t(H, W) = 0$, and $\EE[h(W)]
  = \frac12 h(0)$, which is strictly less than $h(r)$ for all $r < 0.186$.
\end{proof}

\section{Open problems} \label{sec:end}

We conclude with some open problems concerning the
variational problem for upper and lower tails.

\begin{itemize}
\item \textbf{Upper tail phase diagram.} Determine the upper tail
  replica symmetry phase diagram for non-regular $H$.
\item \textbf{Lower tail phase diagram.} Determine the lower tail
  replica symmetry phase diagram for $K_3$, and more generally for any
  non-bipartite graph $H$. In particular, determine $r_H^*$ from
  Conjecture~\ref{conj:h}. For a bipartite graph $H$, determine
  whether there is replica symmetry everywhere
  (Conjecture~\ref{conj:LT-bipartite}).
\item \textbf{Solution in the symmetry breaking
    phase.} Solve the variational
  problem $\UT$ or $\LT$ at any non-trivial point where the constant graphon is
  not a minimizer.
\end{itemize}

\medskip

\subsection*{Acknowledgments} I would like to thank Eyal Lubetzky for helpful
discussions.

\bibliographystyle{abbrv}
\bibliography{lt_ref}

\end{document}